\documentclass[a4paper,12pt]{article}
\usepackage{amsfonts}
\usepackage{amscd,color}
\usepackage{amsmath,amsfonts,amssymb,amscd}
\usepackage{indentfirst,graphicx,epsfig}
\usepackage{graphicx,psfrag}
\input{epsf}

\newtheorem{thm}{Theorem}[section]

\newtheorem{lem}{Lemma}[section]

\newtheorem{cor}{Corollary}[section]

\newenvironment {proof} {\noindent{\em Proof.}}{\hspace*{\fill}$\Box$\par\vspace{4mm}}

\def\qed{\hfill \nopagebreak\rule{5pt}{8pt}}

\title{Colorful monochromatic connectivity\\ of random graphs\footnote{Supported by NSFC No.11371205 and PCSIRT.}}

\author{Ran Gu, Xueliang Li, Zhongmei Qin\\
{\small  Center for Combinatorics and LPMC-TJKLC}\\
{\small Nankai University, Tianjin 300071, P.R. China}\\
{\small Email: guran323@163.com, lxl@nankai.edu.cn, qinzhongmei90@163.com}\\
}
\date{}
\begin{document}
\maketitle
\begin{abstract}
An edge-coloring of a connected graph $G$ is called a {\it
monochromatic connection coloring} (MC-coloring, for short),
introduced by Caro and Yuster, if there is a monochromatic path
joining any two vertices of the graph $G$. Let $mc(G)$ denote the
maximum number of colors used in an MC-coloring of a graph $G$. Note
that an MC-coloring does not exist if $G$ is not connected, and in
this case we simply let $mc(G)=0$. We use $G(n,p)$ to denote the
Erd\"{o}s-R\'{e}nyi random graph model, in which each of the
$\binom{n}{2}$ pairs of vertices appears as an edge with probability
$p$ independently from other pairs. For any function $f(n)$
satisfying $1\leq f(n)<\frac{1}{2}n(n-1)$, we show that if $\ell n
\log n\leq f(n)<\frac{1}{2}n(n-1)$ where $\ell\in \mathbb{R}^+$,
then $p=\frac{f(n)+n\log\log n}{n^2}$ is a sharp threshold function
for the property $mc\left(G\left(n,p\right)\right)\ge f(n)$; if
$f(n)=o(n\log n)$, then $p=\frac{\log n}{n}$ is a sharp threshold
function for the property
$mc\left(G\left(n,p\right)\right)\ge f(n)$. \\[2mm]
\end{abstract}

\section{Introduction}

All graphs in this paper are undirected, finite and simple. We
follow \cite{BM} for graph theoretical notation and terminology not
defined here. Let $G$ be a nontrivial connected graph with an {\it
edge-coloring} $c : E(G)\rightarrow \{1, 2, \ldots, t\}, \ t \in
\mathbb{N}$, where adjacent edges may have the same color. A path of
$G$ is said to be a {\it rainbow path} if no two edges on the path
have the same color. A connected graph is {\it rainbow connected} if
there is a rainbow path connecting any two vertices. An
edge-coloring of a connected graph is called a {\it rainbow
connection coloring} if it makes the graph rainbow connected. The
concept of rainbow connection of graphs was introduced by Chartrand
et al. in \cite{CJMZ}. The rainbow connection number of a connected
graph $G$, is the smallest number of colors that are needed in order
to make $G$ rainbow connected. Recently, the rainbow connection
colorings have been well-studied, and for details we refer to
\cite{LS, LSS}.

In 2011, Caro and Yuster \cite{CY} introduced a natural counterpart
question of rainbow connection colorings, which is called the
monochromatic connection coloring. An edge-coloring of a connected
graph $G$ is called a {\it monochromatic connection coloring}
(MC-coloring, for short) if there is a monochromatic path joining
any two vertices. Let $mc(G)$ denote the maximum number of colors
used in an MC-coloring of a graph $G$, which called the {\it
monochromatic connection number} of $G$. Note that an MC-coloring
does not exist if $G$ is not connected, and in this case we simply
let $mc(G)=0$. Denote by $n$ and $m$ the number of vertices and
edges of graph $G$, respectively. Note that by simply coloring the
edges of a spanning tree of $G$ with one color, and assigning the
remaining edges other distinct colors, we obtain an MC-coloring of
$G$, and this MC-coloring provides a straightforward lower bound for
$mc(G)$, which is summarized that as a theorem below.
\begin{thm}\label{low}
For any connected graph $G$, $mc(G)\ge m-n + 2$.
\end{thm}
In particular, $mc(G)=m-n+2$ whenever $G$ is a tree.  Caro and
Yuster \cite{CY} also showed that there are dense graphs that still
meet this lower bound.

\begin{thm} \cite{CY}
Let $G$ be a connected graph with $n>3$. If $G$ satisfies any of the
following properties, then $mc(G)=m-n+2$.

(a) $\overline{G}$ (the complement of $G$) is 4-connected.

(b) $G$ is triangle-free.

(c) $\Delta(G)< n-\frac{2m-3(n-1)}{n-3}$. In particular, this holds
if $\Delta(G) \leq (n+1)/2$, and also holds if $\Delta(G) \leq
n-2m/n $.

(d) The diameter of $G$ is at least 3.

(e) $G$ has a cut vertex.
\end{thm}

For the  upper bounds of $mc(G)$, Caro and Yuster \cite{CY} gave the
following result:

\begin{thm}\cite{CY}
Let $G$ be a connected graph. Then

(a) $mc(G) \leq m-n+\chi(G)$, where $\chi(G)$ is the vertex
chromatic number of $G$.

(b) if $G$ is not $r$-connected, then $mc(G) \leq m-n+r$.
\end{thm}

In this paper, we study the number $mc(G)$ for random graphs. The
most frequently occurring probability model of random graphs is the
{\it Erd\"{o}s-R\'{e}nyi random graph model} $G(n,p)$ \cite{ER}.
The model $G(n,p)$ consists of all graphs with $n$ vertices in which
the edges are chosen independently and with probability $p$. We say
an event $\mathcal{A}$ happens \textit{with high  probability} if
the probability that it happens approaches $1$ as $n\rightarrow
\infty $, i.e., $Pr[\mathcal{A}]=1-o_n(1)$. Sometimes, we say
\textit{w.h.p.} for short. We will always assume that $n$ is the
variable that tends to infinity.

Let $G$, $H$ be two graphs on $n$ vertices. A property $P$ is said
to be \emph{monotone} if whenever $G\subseteq H$ and $G$ satisfies
$P$, then $H$ also satisfies $P$. For a graph property $P$, a
function $p(n)$ is called a {\it threshold function} of $P$ if:
\begin{itemize}
\item for every $r(n) = \omega(p(n))$, $G(n, r(n))$ w.h.p. satisfies $P$; and

\item for every $r'(n) = o(p(n))$, $G(n, r'(n))$ w.h.p.
does not satisfy $P$.
\end{itemize}

Furthermore, $p(n)$ is called a sharp threshold function of $P$ if
there exist two positive constants $c$ and $C$ such that:
\begin{itemize}
  \item for every $r(n) \geq C\cdot p(n)$, $G(n, r(n))$ w.h.p. satisfies $P$; and
  \item for every $r'(n) \leq c\cdot p(n)$, $G(n, r'(n))$ w.h.p.
does not satisfy $P$.
\end{itemize}

In the extensive study of the properties of random graphs, many
researchers observed that there are sharp threshold functions for
various natural graph properties. It is well-known that all monotone
graph properties have sharp threshold functions; see \cite{BT} and
\cite{FK}. For the property $rc(G(n, p)) \le 2$, Caro et al.
\cite{Caro} proved that $p = \sqrt{\log n/n}$ is the sharp threshold
function. He and Liang \cite{HL} studied further the rainbow
connectivity of random graphs. Specifically, they obtained that
$(\log n)^{(1/d)}/n^{(d-1)/d}$ is the sharp threshold function for
the property $rc(G(n, p)) \le d$,  where $d$ is a constant.

For the monochromatic connectivity of a graph, one aims to find as
many colors as possible to keep the graph monochromatically
connected. Also, it is natural to ask what kind of graphs have large
$mc(G)$. That is, we can use a great many colors to make the graph
monochromatically connected. Furthermore, what will happen if we
require the number of colors to relate with the order of the graph ?
So it is interesting to consider the threshold function of the
property $mc\left(G\left(n,p\right)\right)\ge f(n)$, where $f(n)$ is
a function of $n$. For any graph $G$ with $n$ vertices and any
function $f(n)$, having $mc(G) \ge f(n)$ is a monotone graph
property (adding edges does not destroy this property), so it has a
sharp threshold function. Realize that for the sharp threshold
function for the rainbow connectivity of random graphs, the known
results all require that the number of colors is independent of the
order of the random graph, but our result dose not have that
restriction. Our main result is as follows.
\begin{thm}\label{thm1}
Let $f(n)$ be a function satisfying $1\leq f(n)<\frac{1}{2}n(n-1)$. Then
\begin{equation*}
p=
\left\{
  \begin{array}{ll}
   \frac{f(n)+n\log\log n}{n^2} & \hbox{ if $\ell n \log n\leq f(n)
   <\frac{1}{2}n(n-1)$, where $\ell\in \mathbb{R}^+$,} \\
   \frac{\log n}{n} & \hbox{ if  $f(n)=o(n\log n)$.} \end{array}
\right.
\end{equation*}
is a sharp threshold function for the property
$mc\left(G\left(n,p\right)\right)\ge f(n)$.
\end{thm}
{\bf Remark.} Note that $mc\left(G\left(n,p\right)\right)\le
\frac{1}{2}n(n-1)$ for any probability function $0\le p\le 1$, and
$mc\left(G\left(n,p\right)\right)= \frac{1}{2}n(n-1)$ if and only if
$G(n, p)$ is isomorphic to the complete graph $K_n$. Hence we only
concentrate on the case $f(n)<\frac{1}{2}n(n-1)$.

\section{Proof of Theorem \ref{thm1}}

In \cite{CY}, Caro and Yuster gave the following upper bound for
$mc(G)$.
\begin{thm}\label{upper}
If the minimum degree of $G$ is $\delta(G) = s$, then $mc(G)\le
|E(G)|-|V(G)| + s+1$.
\end{thm}

In this paper, we use the following version of  Chernoff bound:

\begin{lem}\label{lem1}\cite{AS} \textbf{(Chernoff Bound)}
If $X$ is a binomial random variable with expectation $\mu$, and
$0<\delta<1$, then
\[\Pr [X < (1 - \delta )\mu ] \le \exp \left( { - \frac{{{\delta ^2}\mu }}{2}} \right)\]
and if $\delta>0$, then
\[\Pr [X > (1 + \delta )\mu ] \le \exp \left( { - \frac{{{\delta ^2}\mu }}{2+\delta}} \right).\]
\end{lem}

Throughout the paper ``log" denotes the natural logarithm. The
following theorem is a classical result on the connectedness of a
random graph.
\begin{thm}\label{ppt}\cite{ER}
Let $p=(\log n +a)/n$. Then
\begin{equation*}
Pr[G(n,p)\  is \ connected)]\rightarrow
\left\{
  \begin{array}{ll}
   e^{-e^{-a}} & \hbox{ if $|a|=O(1)$,} \\
    0 & \hbox{  $a\rightarrow -\infty$}, \\
    1 & \hbox{ $a\rightarrow +\infty$.}
  \end{array}
\right.
\end{equation*}
\end{thm}

From Theorem \ref{ppt} and the definition of sharp threshold
functions, we can derive the following corollary immediately.
\begin{cor}\label{cor}
$p=\frac{\log n}{n}$ is a sharp threshold function for $G(n,p)$ to
be connected.
\end{cor}

Now we prove Theorem \ref{thm1}. According to the range of $f(n)$,
we have the following two cases.

\noindent{\bf Case 1.} $\ell n \log n\leq f(n)<\frac{1}{2}n(n-1)$,
where $\ell\in \mathbb{R}^+$.

To establish a sharp threshold function for a graph property, the
proof should be two-folds. We first show one direction.

\begin{thm}\label{t2}
There exists a constant $C$ such that
$mc\left(G\left(n,C\frac{f(n)+n\log\log n}{n^2}\right)\right)\ge
f(n)$ w.h.p. holds.
\end{thm}
\begin{proof}
Let
\begin{equation*}
C=
\left\{
  \begin{array}{ll}
   5 & \hbox{ if $\ell\geq 1$} \\
   \frac{5}{\ell} & \hbox{ if $0<\ell<1$}\\
  \end{array}
\right.
\end{equation*}
and $p=\frac{f(n)+n\log\log n}{n^2}$. By Theorem \ref{ppt}, it is
easy to check that $G\left(n,Cp\right)$ is w.h.p. connected. Let
$\mu_1$ be the expectation of the number of edges in
$G\left(n,Cp\right)$. So $$\mu_1=\frac{n(n-1)}{2}\cdot
Cp=\frac{C}{2}\left(\frac{n-1}{n}f(n) +(n-1)\log\log n\right).$$
From Lemma \ref{lem1}, we have
$$\Pr [|E(G(n,Cp))| < \frac{\mu_1}{2} ] \le \exp
\left( { - \frac{1}{2}\cdot\frac{1}{4}\mu_1} \right)
=\exp \left( { - \frac{1}{8}\mu_1}\right)=o(1).$$
Note that if $|E(G(n,Cp))| \ge \frac{\mu_1}{2} $, then by Theorem
\ref{low}, we have that
\begin{align*}
mc\left(G\left(n,Cp\right)\right)&\ge |E(G(n,Cp))|-n+2\\
                                 &\ge  \frac{\mu_1}{2}-n+2\\
                                 &=\frac{C}{4}\left(\frac{n-1}
                                 {n}f(n)+(n-1)\log\log n\right)-n+2\\
                                 &\ge\frac{5}{4}\left(\frac{n-1}
                                 {n}f(n)+(n-1)\log\log n\right)-n+2\\
                                 &\ge f(n),
\end{align*}
for $n$ sufficiently large. Thus, we obtain that with probability at
least $1-\exp \left( { - \frac{1}{8}\mu_1}\right)=1-o(1)$,
$mc\left(G\left(n,Cp\right)\right)\ge f(n)$ holds.
 \end{proof}

Next we show the other direction.

\begin{thm}\label{t3}
$mc\left(G\left(n,\frac{f(n)+n\log\log n}{n^2}\right)\right)< f(n)$
w.h.p. holds.
\end{thm}
\begin{proof}
Let $p=\frac{f(n)+n\log\log n}{n^2}$ and $\mu_2$ be the expectation
of the number of edges in $G\left(n,p\right)$. We have
$$\mu_2=\frac{n(n-1)}{2}\cdot
p=\frac{1}{2}\left(\frac{n-1}{n}f(n)+(n-1)\log\log n\right).$$ We
obtain that
$$\Pr [|E(G(n,p))| > \frac{3}{2}\mu_2 ] \le \exp
\left( { - \frac{\frac{1}{4}\mu_2}{2+\frac{1}{2}}} \right) =\exp
\left( { - \frac{1}{10}\mu_2}\right)=o(1)$$ by Lemma \ref{lem1}. If
$G(n,p)$ is not connected, then $mc\left(G\left(n,p\right)\right)=0<
f(n)$. If $G(n,p)$ is connected, let $d$ denote the minimum degree
of $G(n,p)$, it is obvious that $d<n$. If $|E(G(n,p))| \le
\frac{3}{2}\mu_2 $, then from  Theorem \ref{upper}, we have  that
\begin{align*}
mc\left(G\left(n,p\right)\right)&\le |E(G(n,p))|-n+d+1\\
                                 &\le  \frac{3}{2}\mu_2-n+d+1\\
                                 &=\frac{3}{4}\left(\frac{n-1}{n}f(n)
                                 +(n-1)\log\log n\right)-n+d+1\\
                                 &<\frac{3}{4}\left(\frac{n-1}{n}f(n)
                                 +(n-1)\log\log n\right)-n+n+1\\
                                 &< f(n).
\end{align*}
Hence, we have that with probability at least $1-\exp
\left( { - \frac{1}{10}\mu_2}\right)=1-o(1)$,
$mc\left(G\left(n,p\right)\right)< f(n)$ holds.
 \end{proof}
\noindent{\bf Case 2.} $f(n)=o(n\log n)$ or $f(n)$ is a constant.

By Corollary \ref{cor} we have that there exist two positive
constants $c_1$ and $c_2$ such that: for every $r(n) \geq c_1\cdot
p$, $G(n, r(n))$ is w.h.p. connected; and for every $r'(n) \leq
c_2\cdot p$, $G(n, r'(n))$ is w.h.p. not connected. Moreover, for
$r(n) \geq c_1\cdot p$, $|E(G(n, r(n)))|=O(n\log n)$ by Lemma
\ref{lem1}. Hence, $mc(G(n, r(n)))\geq |E(G(n, r(n)))|-n+2\geq
f(n)$. On the other hand, since $G(n, r'(n))$ is w.h.p. not
connected, for every $r'(n) \leq c_2\cdot p$, $mc(G(n, r'(n)))=0<
f(n)$ w.h.p. holds.

Combining Case 1 and Case 2, our main result follows.
 \qed\\

\end{document}